\documentclass[ametsoc,a4paper]{article}
\usepackage{amsmath}
\usepackage{amsthm}
\usepackage{array}
\usepackage{url}

\usepackage{xcolor}
\usepackage{sectsty}
\usepackage{color}
\usepackage[colorlinks=false, linkcolor=blue]{hyperref}
\hypersetup{urlcolor=blue}
\usepackage{authblk}
\title{Sections and referencing}
\usepackage{multirow}
\usepackage{dblfloatfix}
\newtheorem{thm}{Theorem}[section]
\newtheorem{lem}{Lemma}[section]
\newtheorem*{remark}{Remark}
\makeatletter
\newtheoremstyle{break}
   {\topsep}{\topsep}%
   {\itshape}{}%
   {\bfseries}{}%
   {\newline}
   {\thmname{#1}\thmnumber{\@ifnotempty{#1}{ }\@upn{#2}}%
    \thmnote{ {\bfseries(#3)}}}%
\makeatother

\makeatletter
\def\ps@pprintTitle{%
   \let\@oddhead\@empty
   \let\@evenhead\@empty
   \let\@oddfoot\@empty
   \let\@evenfoot\@oddfoot
}
\makeatother

\makeatletter
\renewcommand\@biblabel[1]{}
\renewenvironment{thebibliography}[1]
     {\section*{\refname}%
      \@mkboth{\MakeUppercase\refname}{\MakeUppercase\refname}%
      \list{}%
           {\leftmargin0pt
            \@openbib@code
            \usecounter{enumiv}}%
      \sloppy
      \clubpenalty4000
      \@clubpenalty \clubpenalty
      \widowpenalty4000%
      \sfcode`\.\@m}
     {\def\@noitemerr
       {\@latex@warning{Empty `thebibliography' environment}}%
      \endlist}
\makeatother

\begin{document}
\begin{center}
{\bf \textsc{Upper and Lower Class Functions for Maximum Likelihood Estimator for Single server Queues
}}\\

\vskip 1cm
{ {\bf Sarat Kumar Acharya} and {\bf Saroja Kumar Singh}\\
(acharya\_sarat@yahoo.co.in, sarojasngh@gmail.com)}

 \vskip .5cm
   \small{P. G. Department of Statistics,
   Sambalpur University, Odisha, India}

\vskip .5cm
\end{center}

\begin{abstract}
Upper and  lower class functions for  the maximum  likelihood estimator
of the arrival and the service rates in a $GI/G/1$ queue are studied and
the results are verified for $M/M/1$ queue.
\end{abstract}

\vskip .5cm

\textbf{Key Words:} Single server queues,  $GI/G/1$ queue,  Exponential  families,
Maximum  likelihood estimator, Asymptotic inference, Asymptotic normality.

\vskip .5cm
{\bf{AMS 2010 subject classifications}} 60K25, 68M20, 62F12
\vskip .5cm

\section{Introduction}\label{sec:1}
Statistical inference is an integral part in any use of queueing models
in decision making. Not much works seems to have been done in this line
of research. The earliest work in this direction seems to be that of
Clarke (1957), who obtained the maximum likelihood estimators(MLE) of
the parameters in an M/M/1 queue in equilibrium. Cox (1965) and Wolff (1965)
carried the investigation with several ideas. The papers by Benes (1957)
and Goyal and Harris (1972) are also worth mentioning. A detail survey
of the earlier works in this direction is given in Bhat and Rao (1987).
Basawa and Prabhu (1988) have studied the asymptotic inference for single
server queues and have proved the consistency and asymptotic normality of
the maximum likelihood estimators. Acharya (1999) has studied the rate of
convergence of the distribution of the maximum likelihood estimators of
the arrival and the service rates from a single server queue.

Let $h(t)$ be a non-negative, non-decreasing function increasing to infinity.
We say that $h(t)$  belongs to the upper class or lower class of a
stochastic process $\{Y(t), 0\leq t \leq T\}$ according as
$$ P\{Y(t) > h(t) ~i.o. ~ \text{as} ~ t \rightarrow \infty \}= 0 ~ or ~ 1.$$

The purpose of this paper is to study the upper and lower class functions for
the difference between the maximum likelihood estimators and the true values
of the arrival and the service rates from a single server queue. In Theorem
\ref{theorem3.1} the integral test criteria for the upper and lower class
functions for the maximum likelihood estimators is developed.
Theorem \ref{theorem3.2} deals with the characterization problem for
the upper and lower class functions for the estimators.

\section{The Maximum Likelihood Estimator}\label{sec:2}
Consider a single server queueing system in which the
interarrival times $\{u_{k}, k\geq 1\}$ and the service times $\{v_{k}, k\geq 1\}$ are
two independent sequences of independent and identically distributed
nonnegative random variables with densities $f(u; \theta)$ and $g(v; \phi)$, respectively, where
$\theta$and $\phi$ are unknown parameters. Let us assume that $f$
and $g$ belong to the continuous exponential families given by
\begin{equation}\label{eqn:1}
  f(u; \theta)= a_{1}(u) e^{[\theta h_{1}(u)- k_{1}(\theta)]},
\end{equation}

\begin{equation}\label{eqn:2}
  g(v; \phi) = a_{2}(v) e^{[\phi h_{2}(v)- k_{2}(\phi)]}.
\end{equation}
It is further assumed that the densities in \eqref{eqn:1} and in \eqref{eqn:2} are
equal to zero on $(-\infty, 0)$.

For simplicity we assume that the initial customer arrives at
time $t=0$. Our sampling scheme is to observe the system over a
continuous time interval $[0, T]$ where  $T$ is a suitable
stopping time. The sample data consist of
\begin{equation}\label{eqn:3}
\{A(T), D(T), u_{1}, u_{2}, u_{3},\cdots \cdots, u_{A(T)},  v_{1},
v_{2},\cdots \cdots, v_{D(T)} \},
\end{equation}
where $A(T)$ is the number of arrivals and $D(T)$ is the number
of departures during $(0, T]$. Obviously no arrivals occur
during $[\sum_{i=1}^{A(T)} u_{i}, T]$ and no departures during
$[\gamma(T)+\sum_{i=1}^{D(T)}v_{i}, T]$, where $\gamma(T)$ is
the total idle period in $(0,T]$.

Some possible stopping rules to determine $T$ are given below:

Rule 1.  Observe the system until a fixed time $t$. Here $T=t$
with probability one and $A(T)$ and $D(T)$ are both random
variables.\\

Rule 2.  Observe the system until $d$ departures have occurred so
that $D(T)=d$. Here $T= \gamma(T)+ v_{1}+v_{2}+\cdots+v_{d}$ and
$A(T)$ are random variables.\\

Rule 3.  Observe the system until $m$ arrivals take place so
that $A(T)=m$. Here $T=u_{1}+u_{2}+u_{3}+\cdots+u_{m}$ and
$D(T)$ are random variables.\\

Rule 4.  Stop at the $nth$ transition epoch. Here, $T, A(T)$ and
$D(T)$ are all random variables and $A(T)+D(T)=n$.\\

Under rule 4, we stop either with an arrival or in a departure.
If we stop with an arrival, then $\sum_{i=1}^{A(T)} u_{i}=T$ and
no departures during $[\gamma(T)+\sum_{i=1}^{D(T)}v_{i}, T]$.
Similarly, if we stop in a departure, then
$\gamma(T)+\sum_{i=1}^{D(T)}v_{i}=T$ and there are no arrivals
during $[\sum_{i=1}^{A(T)} u_{i}, T]$.

The likelihood function based on data \eqref{eqn:3} is given by
\begin{align}\label{eqn:4}
  L_{T}(\theta, \phi) & = \prod_{i=1}^{A(T)} f(u_{i},\theta)\prod_{i=1}^{D(T)} f(v_{i},\phi) \nonumber \\
  & \times\Biggr[1-F_{\theta}[T-\sum_{i=1}^{A(T)}
  u_{i}]\Biggr]\Biggr[1-G_{\phi}[T-\gamma(T)-\sum_{i=1}^{D(T)}v_{i}]\Biggr],
\end{align}
where $F$ and $G$ are distribution functions corresponding to the
densities $f$ and $g$ respectively. The likelihood function
$L_{T}(\theta,\phi)$ remains valid under all the stopping rules.

The approximate likelihood $L_{T}^{a}(\theta,\phi)$ is defined as
\begin{equation}\label{eqn:5}
L_{T}^{a}(\theta,\phi)= \prod_{i=1}^{A(T)}
f(u_{i},\theta)\prod_{i=1}^{D(T)} f(v_{i},\phi), \quad  \quad \text{(cf. Basawa and Prabhu (1988))}.
\end{equation}
Under certain conditions the maximum likelihood estimates obtained from
\eqref{eqn:5} are asymptotically equivalent to those obtained from \eqref{eqn:4}(cf. Basawa and Prabhu (1988)).

We assume that the following condition holds: \\
\textbf{Condition C1:}
Suppose that there exists a positive function $\varepsilon(T)\downarrow 0$ such that
$T \varepsilon^2(T) \rightarrow \infty$ as $T \rightarrow \infty$, and
$$ P \bigg\{ \bigg|\frac{A(T)}{E(A(T))} - 1\bigg| \geq \varepsilon(T) \bigg\}= O(\varepsilon^{\frac{1}{2}}(T))$$
and
$$ P \bigg\{ \bigg|\frac{D(T)}{E(D(T))} - 1 \bigg| \geq \varepsilon(T)  \bigg\}= O(\varepsilon^{\frac{1}{2}}(T))$$
Basawa and Prabhu (1988) have shown that the maximum likelihood estimator of $\theta$ and $\phi$ are given by
\begin{eqnarray}\label{eqn:6}
\hat{\theta}_T=
\eta_{1}^{-1} \bigg[(A(T))^{-1}\sum_{i=1}^{A(T)}h_{1}(u_{i}) \bigg], \\
\hat{\phi}_T= \eta_{2}^{-1}\bigg[(D(T))^{-1}\sum_{i=1}^{D(T)}h_{2}(v_{i}) \bigg]
\end{eqnarray}
where $\eta_i^{-1}(.)$ denotes the inverse functions of $\eta_i(.)$ for $i=1,2$ and
\begin{align}
  \eta_1(\theta)=E(h_1(u)) &= k_1^{'}(\theta) \\
  \text{and}  \nonumber \\
  \eta_2(\phi)=E(h_2(v)) &= k_2^{'}(\phi)
\end{align}
The Fisher information matrix is given by
\begin{equation}
I(\theta, \phi) = \left[
        \begin{array}{cc}
          \sigma_1^2E(A(T)) & 0  \\
          0 &  \sigma_2^2E(D(T)) \\
        \end{array}
      \right]
      = \left[
        \begin{array}{cc}
          I(\theta) & 0  \\
          0 &  I(\phi) \\
        \end{array}
      \right],
\end{equation}
where $\sigma_1^2=\sigma_1^2(\theta)=var_{\theta}(h_1(u))$ and $\sigma_2^2=\sigma_2^2(\theta)=var_{\phi}(h_2(v))$.

Under suitable stability conditions on stopping times, Basawa and Prabhu (1988) have proved that
\begin{equation}\label{eqn:11}
  \hat{\theta}_T \rightarrow \theta_0 \quad \text{and} \quad \hat{\phi}_T \rightarrow \phi_0 \quad \text{as} \quad T \rightarrow \infty
\end{equation}
and
\begin{equation}
I^{1/2}(\theta_0, \phi_0)
\left[
        \begin{array}{c}
          \hat{\theta}_T - \theta_0  \\
           \hat{\phi}_T - \phi_0\\
        \end{array}
      \right]
      \Rightarrow N \left[
      \left(
        \begin{array}{c}
          0  \\
          0  \\
        \end{array}
        \right),
        \left(
        \begin{array}{cc}
          1 & 0  \\
          0 &  1 \\
        \end{array}
      \right)
      \right],
\end{equation}
where $\theta_0$ and $\phi_0$ denote the true value of $\theta$ and
$\phi$ respectively, and the symbol $\Rightarrow$ denotes the convergence in distribution.

The likelihood function in \eqref{eqn:5} becomes
\begin{equation}\label{eqn:13}
  L^a_T(\theta, \phi)= \prod \limits_{i=1}^{A(T)}a_1(u_i) \prod \limits_{i=1}^{D(T)}a_2(v_i)
exp\bigg \{ \sum\limits_{i=1}^{A(T)}[\theta h_1(u_i) - k_1(\theta)] + \sum\limits_{i=1}^{D(T)}[\phi h_2(v_i) - k_2(\phi)] \bigg\}
\end{equation}
and the log likelihood function is
\begin{equation}\label{eqn:14}
  l(\theta, \phi)= log \bigg[ \prod \limits_{i=1}^{A(T)}a_1(u_i) \prod \limits_{i=1}^{D(T)}a_2(v_i) \bigg] +
 \sum\limits_{i=1}^{A(T)}[\theta h_1(u_i) - k_1(\theta)] + \sum\limits_{i=1}^{D(T)}[\phi h_2(v_i) - k_2(\phi)]
\end{equation}
Let
$$
l^{'}(\theta_0)=\frac{\partial}{\partial \theta} l(\theta, \phi)\bigg|_{\theta=\theta_0}, \quad
l^{''}(\theta_0)=\frac{\partial^2}{\partial \theta^2} l(\theta, \phi)\bigg|_{\theta=\theta_0}
$$
Similarly
$l^{'}(\hat{\theta}_T)$, $l^{'}(\hat{\phi}_T)$, $l^{'}(\phi_0)$, $l^{''}(\phi_0)$, $l^{''}(\hat{\theta}_T)$ and $l^{''}(\hat{\phi}_T)$
are defined.

Now by Taylor's formula
\begin{equation}\label{eqn:15}
  l^{'}(\theta_0)=l^{'}(\hat{\theta}_T) + (\theta_0-\hat{\theta}_T) \{ l^{''}(\theta_0)
  +[l^{''}(\bar{\theta}_T)-l^{''}(\theta_0)]\}
\end{equation}
and
\begin{equation}\label{eqn:16}
  l^{'}(\phi_0)=l^{'}(\hat{\phi}_T) + (\phi_0-\hat{\phi}_T) \{ l^{''}(\phi_0)
  +[l^{''}(\bar{\phi}_T)-l^{''}(\phi_0)]\}
\end{equation}
where $|\theta_0-\bar{\theta}_T| \leq |\theta_0-\hat{\theta}_T| $ and
$|\phi_0-\bar{\phi}_T| \leq |\phi_0-\hat{\phi}_T|$.

Since $\hat{\theta}_T$ and $\hat{\phi}_T$ are the MLEs of $\theta$ and $\phi$ respectively, we get
\begin{equation}\label{eqn:17}
  \hat{\theta}_T-\theta_0=\frac{l^{'}(\theta_0)}{[l^{''}(\theta_0) - l^{''}(\bar{\theta}_T)]-l^{''}(\theta_0)}
\end{equation}
and
\begin{equation}\label{eqn:18}
  \hat{\phi}_T-\phi_0=\frac{l^{'}(\phi_0)}{[l^{''}(\phi_0) - l^{''}(\bar{\phi}_T)]-l^{''}(\phi_0)}
\end{equation}

\textbf{Condition C2:}
There exists a positive function $\varepsilon(.)$ such that
$$
\int\limits_{0}^{T}t^{-1}(loglogt)(\varepsilon^{1/2}(t))dt < \infty.
$$
Under condition (C1) Acharya (1999) has shown that
\begin{equation}\label{eqn:19}
  \sup_x\bigg| P\{I^{1/2}(\hat{\theta}_T- \theta_0)\leq x \} -\Phi(x)\bigg|=O(\varepsilon^{1/2}(T))
\end{equation}
and
\begin{equation}\label{eqn:20}
  \sup_x\bigg| P\{I^{1/2}(\hat{\phi}_T- \phi_0)\leq x \} -\Phi(x)\bigg|=O(\varepsilon^{1/2}(T))
\end{equation}

In section \ref{sec:3} we will state and prove results only for the arrival process and write $I$ for $I(\theta_0)$.
The corresponding results for the departure case are similar.
\section{Upper class and lower class functions for the MLE:}\label{sec:3}
\begin{lem}\label{lemma3.1}
Let  $h(t)$ be a monotonically increasing function of $t$. Then for a sequence $\{ t_n \} \uparrow \infty $  as $n\rightarrow \infty$,
$$
\sum \limits_{n=1}^{\infty}(loglogt_n) t_n^{-1}P\{I^{1/2}(\hat{\theta_{t_n}} - \theta_0) > h(t_n) \}
$$
and
$$
\sum\limits_{n=1}^{\infty}(loglogt_n) (t_n h(t_n))^{-1} exp (-h^2(t_n)/2)
$$
converge or diverge simultaneously.
\end{lem}

\begin{lem}\label{lemma3.2}
Let $h(t_n)$ be a positive monotonically increasing function which increases to infinity. Then
$$
\sum\limits_{n=1}^{\infty}(t_n h(t_n))^{-1} exp(-h^2(t_n)/2)
$$
and
$$
\sum\limits_{n=1}^{\infty}(loglogt_n) (t_n h(t_n))^{-1} exp \{(-h^2(t_n)/2) (1+C/loglogt_n)\}
$$
converge or diverge simultaneously.
\end{lem}

\begin{thm}\label{theorem3.1}
Let $h(t)$ be a positive monotonically increasing function which increases to infinity. Then
$$ P\{I^{1/2}(\hat{\theta}_T - \theta_0) > h(T) ~i.o. ~ as ~ T \rightarrow \infty \}= 0 ~ or ~ 1,$$
according as
$$
I(h(T))=\int\limits_{1}^{\infty}(h(T)/T) exp\{ -h^2(T)/2 \} dT
$$
converges or diverges.
\end{thm}

\begin{remark}
\begin{enumerate}
\item From Theorem \ref{theorem3.1}, it is observed that for a sequence $\{ t_n\}\uparrow \infty$, the function $h(t_n)$
belongs to the upper class or lower class of $I^{1/2}(\hat{\theta}_{t_n}-\theta_0)$ according as \\
$$
\sum\limits_{n=1}^{\infty}(t_n h(t_n))^{-1} exp(-h^2(t_n)/2)
$$
converges or diverges.

\item  By applying the Lemma \ref{lemma3.2}, it is observed that the two series \\
$$
\sum\limits_{n=1}^{\infty}(loglogt_n) (t_n h(t_n))^{-1} exp \{(-h^2(t_n)/2) (1+C/loglogt_n)\}
$$  \\
and
$$
\sum\limits_{n=1}^{\infty} (loglogt_n)(t_nh(t_n))^{-1} exp \{(-h^2(t_n)/2) \}
$$
converge or diverge simultaneously $if~ and ~ only ~if$ $h(t_n)$ belongs to the upper or lower class of $I^{1/2}(\hat{\theta}_{t_n}-\theta_0)$
respectively.
\end{enumerate}
\end{remark}

\begin{thm}\label{theorem3.2}
Under conditions of Lemma \ref{lemma3.1}, $h(T)$ belongs to the upper class or lower class of
$I^{1/2}(\hat{\theta}_T-\theta_0)$ if and only if
$$
\int\limits_{1}^{\infty}(loglogT) T^{-1}P\{I^{1/2}(\hat{\theta_{T}} - \theta_0) > h(T) \}dt
$$
is convergent or divergent respectively.
\end{thm}

\begin{proof}[Proof of Lemma \ref{lemma3.1}]
Consider a sequence $\{t_n\}\uparrow \infty$. Then for the sequence of MLE $\{\hat{\theta}_{t_n}\}$
$$\sup_x\bigg|P\{I^{1/2}(\hat{\theta}_{t_n}-\theta_0) \leq x \} - \Phi(x) \bigg| = O(\varepsilon^{1/2}(t_n))$$
i.e.
$$\sum\limits_{n=1}^{\infty} (loglogt_n)(t_n)^{-1} exp \{(-h^2(t_n)/2) \}
\bigg|P\{I^{1/2}(\hat{\theta}_{t_n}-\theta_0) > h(t_n) \} - \{1-\Phi(h(t_n)) \} \bigg| < \infty \quad \text{(using C2)}$$
Therefore,\\
$$\sum\limits_{n=1}^{\infty} (loglogt_n)(t_n)^{-1}
\bigg|P\{I^{1/2}(\hat{\theta}_{t_n}-\theta_0) > h(t_n) \} - C h^{-1}(t_n) exp(-h^2(t_n)/2) \bigg| < \infty$$,
where $C$ is some positive constant. Thus
$$
\sum \limits_{n=1}^{\infty}(loglogt_n) t_n^{-1}P\{I^{1/2}(\hat{\theta_{t_n}} - \theta_0) > h(t_n) \}
$$
and
$$
\sum\limits_{n=1}^{\infty}(loglogt_n) (t_n h(t_n))^{-1} exp (-h^2(t_n)/2)
$$
converge or diverge simultaneously.
\end{proof}

For proof of Lemma \ref{lemma3.2}  we refer to Davis (1969).

\begin{proof}[Proof of Theorem \ref{theorem3.1}]
From the log-likelihood equation given in equation
\eqref{eqn:14} we have
\begin{align}
  l'(\theta)& = \sum_{i=1}^{A(T)}h_1(u_i) - A(T)k_1^{'}(\theta) \nonumber \\
  l^{''}(\theta)& =A(T)k_1^{''}(\theta)=-A(T)\sigma_1^2(\theta) \nonumber.
\end{align}
From equation \eqref{eqn:17} we get
\begin{equation}\label{eqn:21rev1}
I^{1/2}(\hat{\theta}_T-\theta_0)=
\frac{l'(\theta_0)/I^{1/2}}{\{[l^{''}(\theta_0)-
l^{''}(\bar{\theta}_T)]-l^{''}(\theta_0)\}/I}
\end{equation}
It is easy to see that
$H_{A(T)}=\sum_{i=1}^{A(T)}h_1(u_i) - A(T)k_1^{'}(\theta_0)$
is a square integrable martingale with zero mean since $\{u_i, i\geq 1\}$
is an independent sequence of random variables.
Hence, by the Skorokhod representation (see Hall and Heyde
[10, appendix I, theorem A.1]), there exists a standard
Brownian motion $W(.)$ and a non-negative random variable
$T_i, ~1\leq i \leq A(T)$, such that
without loss of generality
$$
H_i=W(T_i),~ 1\leq i \leq A(T).
$$
Hence, by Theorem 2.3 of Feigin (1976) due to Kunita and Watenabe,
\begin{equation}\label{eqn:21rev2}
  \frac{H_i}{I^{1/2}}=W(T_i/I) \quad \text{for}~1 \leq i \leq A(T).
\end{equation}

Then from \eqref{eqn:21rev1} with \eqref{eqn:21rev2} we get that
\begin{align}\label{eqn:21}
   & \bigg| \frac{I^{1/2}(\hat{\theta}_T-\theta_0) - \frac{W(T)}{\sqrt{T}}}{(loglogT)^{1/2}}\bigg|   \nonumber  \\
   =& \bigg | \frac{\frac{W(T_{A(T)}) / I^{1/2}}{ \{[l^{''}(\theta_0) -  l^{''}(\bar{\theta}_T)] -l^{''}(\theta_0)\}/I }
-
\frac{W(T)}{\sqrt{T}}}{(loglogT)^{1/2}} \bigg | \nonumber \\
   =& \bigg| \frac{\frac{\sqrt{2} W(T_{A(T)})}{(2T_{A(T)} loglogT_{A(T)})^{1/2}}
\frac{(2T_{A(T)} loglogT_{A(T)})^{1/2}}{(2I loglog I)^{1/2}}
\frac{(2I loglog I)^{1/2} }{(2T loglog T)^{1/2}}
\frac{\sqrt{T}}{\sqrt{I}}}{\frac{l^{''}(\theta_0) -  l^{''}(\bar{\theta}_T)}{I}
-
\frac{l^{''}(\theta_0)}{I}}
-
\frac{\sqrt{2} W(T)}{(2TloglogT)^{1/2}} \bigg|
\end{align}
From the law of iterated logarithm for Brownian motion process
\begin{align}
   & \lim_{T \rightarrow \infty} \frac{W(T)}{(2TloglogT)^{1/2}} =1 \quad a.s. \label{eqn:22}\\
   & \lim_{T \rightarrow \infty} \frac{ W(T_{A(T)})}{(2T_{A(T)} loglogT_{A(T)})^{1/2}} = 1 \quad a.s. \quad
   \text{(cf. Kulinich (1985, P.564))} \\
   & \lim_{T \rightarrow \infty} \frac{(2T_{A(T)} loglogT_{A(T)})^{1/2}}{(2I loglog I)^{1/2}}=1 \\
  & \lim_{T \rightarrow \infty} \frac{(2I loglog I)^{1/2}}{(2T loglog T)^{1/2}} = \lim_{T \rightarrow \infty} \frac{\sqrt{I}}{\sqrt{T}}.
  \end{align}
Moreover,
\begin{align}
 P \bigg\{ \bigg|   \frac{l^{''}(\theta_0)}{I} + 1 \bigg| \geq \frac{\varepsilon(T)}{2}  \bigg \}
=& P \bigg\{ \bigg|   \frac{A(T) \sigma_1^2}{E(A(T)) \sigma_1^2} - 1 \bigg| \geq \frac{\varepsilon(T)}{2}   \bigg \}
=O(\varepsilon^{1/2}(T)) \\
 \quad \quad \quad P\bigg\{ \bigg|   \frac{l^{''}(\theta_0) - l^{''}(\bar{\theta}_T)}{I} \bigg| \geq \frac{\varepsilon(T)}{2}  \bigg \}
=& O(\varepsilon^{1/2}(T)) \quad \text{(cf. Acharya (1999, P.214)} \label{eqn:27}.
\end{align}
Now using \eqref{eqn:22} to \eqref{eqn:27} in \eqref{eqn:21},
\begin{align}
   & \lim_{T \rightarrow \infty} \bigg| \frac{I^{1/2}(\hat{\theta}_T-\theta_0) - \frac{W(T)}{\sqrt{T}}}{(loglogT)^{1/2}}\bigg|=|\sqrt{2} -\sqrt{2}|
  =0 \quad a.s.
\end{align}

The proof of the rest part of the theorem is a direct consequence of Theorem 5.1 of Jain, {\it et al} (1975).
However for completeness, we give a proof of it in short.\\
From \eqref{eqn:27}, for arbitrary $\varepsilon >0$, we have
\begin{equation}\label{eqn:28}
  \bigg| I^{1/2}(\hat{\theta}_T-\theta_0) - \frac{W(T)}{\sqrt{T}} \bigg|
  < \varepsilon (loglogT)^{1/2}  \quad a.s. \quad \text{as} \quad T \rightarrow \infty.
\end{equation}
Let us assume that for all $T$ sufficiently large,
\begin{equation}\label{eqn:29}
h_1(T) \leq h(T) \leq h_2(T),
\end{equation}
where $h_1(T)=(loglogT)^{1/2}$ and $h_2(T)=2(loglogT)^{1/2}$.

If $I(h(T)) < \infty$, then by Kolmogorov's test for Brownian motion motion ( It$\hat{o}$ and H.P. McKean (1974, P. 163)),
we have for any $\varepsilon > 0$,
\begin{equation}\label{eqn:30}
P\{ W(T) > T^{1/2}(h(T)-\varepsilon h(T)) ~ i.o. ~ \text{as} ~ T \rightarrow \infty \} =0,
\end{equation}
since $h(T)-\varepsilon h(T)$ increases as $h(T)$ increases and $I(h(T)-\varepsilon h(T)) < \infty$.

Now using \eqref{eqn:28} in conjunction with \eqref{eqn:29} and \eqref{eqn:30}, we obtain
$$
P\{ I^{1/2}(\hat{\theta}_{T}-\theta_0) > h(T) ~ i.o. ~ \text{as} ~ T \rightarrow \infty \} =0.
$$
On the other hand if $I(h(T))=\infty$, then $I(h(T)+\varepsilon h(T))=\infty$ for every $\varepsilon>0$.

Again $I(h(T)+\varepsilon h(T))=\infty$ is also increasing for sufficiently large $T$ and similar argument shows that
$$
P\{ I^{1/2}(\hat{\theta}_{T}-\theta_0) > h(T) ~ i.o. ~ \text{as} ~ T \rightarrow \infty \} =1.
$$

Now to complete the proof it is sufficient to show that \eqref{eqn:29} may be assumed without any loss of generality.\\
Let $h(T)$ be an arbitrary increasing function. Define
\begin{equation}\label{eqn:31}
  \hat{h}(T)=min[max\{ h(T), h_1(T)\}, h_2(T)].
\end{equation}
By Lemma 2.3 of Jain, {\it et al} (1975), $I(h(T))<\infty$ implies $I(\hat{h}(T))<\infty$ and
$\hat{h} \leq h$ near $\infty$.

Since $\hat{h}$ satisfies \eqref{eqn:29}, we conclude that
 $$
P\{ I^{1/2}(\hat{\theta}_{T}-\theta_0) > \hat{h}(T) ~ i.o. ~ \text{as} ~ T \rightarrow \infty \} =0.
$$

But because $\hat{h} \leq h$ near $\infty$, $I(h(T))<\infty$ implies that
$$
P\{ I^{1/2}(\hat{\theta}_{T}-\theta_0) > h(T) ~ i.o. ~ \text{as} ~ T \rightarrow \infty \} =0.
$$

Again let $I(h(T))=\infty$. Then by Lemma 2.3 of Jain, {\it et al} (1975), $I(\hat{h}(T))<\infty$ and
$$
P\{ I^{1/2}(\hat{\theta}_{T}-\theta_0) > \hat{h}(T) ~ i.o. ~ \text{as} ~ T \rightarrow \infty \} =1.
$$

Hence there exists a sequence $T_n \rightarrow \infty$ such that
\begin{equation}\label{eqn:32}
  I^{1/2}(\hat{\theta}_{T_n}-\theta_0) > \hat{h}(T_n) \quad a.s. \quad \text{for every positive integer} ~ n.
\end{equation}

Since $I(h_2(T)) < \infty$, we have
\begin{equation}\label{eqn:33}
  I^{1/2}(\hat{\theta}_{T_n}-\theta_0) \leq \hat{h}_2(T_n).
\end{equation}
Now from \eqref{eqn:32} and \eqref{eqn:33},
\begin{equation}\label{eqn:34}
  \hat{h}(T_n) < h_2(T_n) \quad \text{for large} ~n.
\end{equation}
Thus from the definition of $\hat{h}$, the inequality \eqref{eqn:34} implies that $h(T_n) \leq \hat{h}(T_n)$ for large $n$
and hence the proof is completed.
\end{proof}

\begin{proof}[Proof of Theorem \ref{theorem3.2}]
Let $h(T)$ belongs to the upper class of $ I^{1/2}(\hat{\theta}_T - \theta_0)$.
Then
$$
\int\limits_{1}^{\infty}(h(T)/T) exp\{ -h^2(T)/2 \} dT < \infty,
$$
i.e. for a sequence $\{ T_n\} \rightarrow \infty$ as $n \rightarrow \infty$,
$$
\sum\limits_{n=1}^{\infty} \frac{h(T_n)}{T_n}) exp\{ -h^2(T_n)/2 \} < \infty,
$$
Therefore, by  Lemma \ref{lemma3.2} and  Remark 2
$$
\sum\limits_{n=1}^{\infty}(loglogT_n) (T_n h(T_n))^{-1} exp \{(-h^2(T_n)/2) \} < \infty
$$
and hence by Lemma \ref{lemma3.1},
$$
\sum \limits_{n=1}^{\infty}(loglogT_n) T_n^{-1}P\{I^{1/2}(\hat{\theta_{T_n}} - \theta_0) > h(T_n) \} < \infty
$$
i.e. $$
\int\limits_{1}^{\infty}(loglogT) T^{-1}P\{I^{1/2}(\hat{\theta_{T}} - \theta_0) > h(T) \}dT < \infty.
$$
To prove sufficiency let
$$
\int\limits_{1}^{\infty}(loglogT) T^{-1}P\{I^{1/2}(\hat{\theta_{T}} - \theta_0) > h(T) \}dT < \infty
$$
i.e. for some sequence $\{T_n\} \uparrow \infty$ as $n \rightarrow \infty$,
$$
\sum \limits_{n=1}^{\infty}(loglogT_n) T_n^{-1}P\{I^{1/2}(\hat{\theta_{T_n}} - \theta_0) > h(T_n) \} < \infty.
$$
Then by Lemma \ref{lemma3.1}
$$
\sum\limits_{n=1}^{\infty}(loglogT_n) (T_n h(T_n))^{-1} exp \{(-h^2(T_n)/2) \} < \infty.
$$
Now by using Remark 2 and the Lemma \ref{lemma3.2} we have
$$
\sum\limits_{n=1}^{\infty} \frac{h(T_n)}{T_n}) exp\{ -h^2(T_n)/2 \} < \infty,
$$
i.e. $$
\int\limits_{1}^{\infty}(h(T)/T) exp\{ -h^2(T)/2 \} dT < \infty,
$$
and hence $h(T)$ belongs to the upper class of $I^{1/2}(\hat{\theta}_T - \theta_0)$.\\
Replacing the convergence statement by divergence in the above proof the result for the lower class is obtained.
\end{proof}

\section{Example}\label{sec:4}
Let us consider the above result for an $M/M/1$ queueing system. Here
$$
f(u, \theta)=\theta e^{-\theta u} \quad \text{and} \quad g(v, \phi)=\phi e^{-\phi v},
$$
so that the loglikelihood function becomes
\begin{align}\label{ex:1}
  l(\theta, \phi) & = log\bigg[ \theta^{A(T)} exp \bigg( -\sum_{i=1}^{A(T) } \theta u_i \bigg) \phi^{D(T)} exp\bigg( -\sum_{i=1}^{D(T)} \phi v_i  \bigg)  \bigg] \\
   & = A(T) log\theta - \sum_{i=1}^{A(T)} \theta u_i + D(T) log\phi - \sum_{i=1}^{D(T)} \phi v_i.
\end{align}
We verify condition C1 as in Acharya (1999). Condition C2 is verified taking $\varepsilon(t)=t^{-\frac{2}{5}}$.
Hence the results of section \ref{sec:3} can be used for this model.

\section{Acknowledgment}
The authors would like to thank the anonymous referees for their helpful comments in improving
the presentation of the paper.

Sarat Kumar Acharya\\
Retd. Professor\\
C-6, Varun Residency\\
Pradhan Para\\
Budharaja\\
Sambalpur-768004\\
Odisha\\
INDIA\\
Email Id.: acharya\_sarat@yahoo.co.in.
\\
\\
Saroja Kumar Singh\\
P. G. Dept. of Statistics\\
Sambalpur University\\
Jyotivihar-768019\\
Sambalpur\\
Odisha\\
INDIA\\
Email Id.: sarojasngh@gmail.com


\begin{thebibliography}{00}
\bibitem{acharya1999}
{\bf Acharya, S. K.} (1999).
On normal approximation for maximum likelihood estimation from single server queues,
{\it Queueing Systems}, {\bf31}(3), 207--216.

\bibitem{basawa1980}
{\bf Basawa, I. V. and Prabhu, N. U.} (1981).
Estimation in single server queues,
{\it Naval Research Logistics Quarterly}, {\bf28}(3), 475--487.

\bibitem{basawa1981}
{\bf Basawa, I. V. and Prabhu, N. U.} (1988).
Large sample inference from single server queues,
{\it Queueing System}, {\bf3}(4), 289--304.


\bibitem{basawa1988}
{\bf Basawa, I. V. and Prakasa, R.} (1980).
{\it Statistical inference for stochastic processes}, Academic Press, London.


\bibitem{benes1957}
{\bf Benes, V. E.} (1957).
A Sufficient Set of Statistics for a Simple Telephone Exchange Model,
{\it Bell System Technical Journal}, {\bf36}(4), 939--964.

\bibitem{bhat1987}
{\bf Bhat, U. N. and Rao, S. S.} (1987).
Statistical Analysis of Queueing Systems,
{\it Queueing System}, {\bf1}(3), 217--247.

\bibitem{clarke1957}
{\bf Clarke, A. B.} (1957).
Maximum Likelihood Estimates in a Simple Queue,
{\it Ann. Math. Statist.}, {\bf28}(4), 1036--1040.

\bibitem{cox1965} {\bf Cox, D. R.} (1965).
Some problems of statistical analysis connected with
congestion, {\it Proc. Of the Symp. On congestion Theory, eds. W. L. Smith and W. E.
Wilkinson} (Univ. of North Carolina Press, Chapel Hill).


\bibitem{davis1969}
{\bf Davis, J. A.} (1969).
A Characterization of the Upper and Lower Classes in Terms
of Convergence Rates,
{\it Ann. Math. Statist.}, {\bf40}(3), 1120--1123.

\bibitem{feigin1976}
{\bf Feigin, P. D.} (1976).
Maximum likelihood estimation for continous
time stochastic processes,
{\it Adv. in Appl. Probab.}, {\bf 9}, 712--736.

\bibitem{goyal1972}
{\bf Goyal, T. L. and Harris, C. M.} (1972).
Maximum Likelihood Estimates for Queues with State-Dependent Service,
{\it Sankhya, Series A}, {\bf34}(1), 65--80.


\bibitem{ito1974}
{\bf It$\hat{o}$, K. and H. P. McKean, J.} (1974).
{\it Diffusion processes and their sample paths},
springer, berlin, heidelberg(reprint).

\bibitem{jain1975}
{\bf Jain, N. C., Jogdeo, K. and Stout, W. F.} (1975).
Upper and Lower Functions for Martingales and Mixing Processes,
{\it Ann. Probab.}, {\bf3}(1), 119--145.

\bibitem{kulinich1985}
{\bf Kulinich, G. L.} (1985).
On the Law of the Iterated Logarithm for OneDimensional Diffusion Processes,
{\it Theory of Probability \& Its Applications},
{\bf29}(3), 563--566.

\bibitem{wolff1965}
{\bf Wolff, R. W.} (1965).
Problems of Statistical Inference for Birth and Death Queuing Models,
{\it Operations Research}, {\bf13}(3), 343--357.
\end{thebibliography}
\end{document}